\documentclass[12pt]{amsart}
\usepackage{latexsym,bbm,amsxtra,amscd,ifthen}
\usepackage{amsfonts}
\usepackage{verbatim}
\usepackage{amsmath}
\usepackage{amsthm}
\usepackage{amssymb}
\usepackage{url}
\usepackage[all]{xy}
\usepackage[colorlinks=true,citecolor=blue,linkcolor=blue]{hyperref}
\usepackage[pagewise]{lineno}

\theoremstyle{plain}

\newtheorem{theorem}{Theorem}
\newtheorem{lemma}[theorem]{Lemma}
\newtheorem{proposition}[theorem]{Proposition}
\newtheorem{corollary}[theorem]{Corollary}

\numberwithin{theorem}{section}
\numberwithin{equation}{theorem}

\theoremstyle{definition}
\newtheorem{definition}[theorem]{Definition}

\newtheorem{example}[theorem]{Example}
\newtheorem{remark}[theorem]{Remark}

\newtheorem*{question*}{Question}

\DeclareMathOperator{\End}{End}

\DeclareMathOperator{\Hom}{Hom}

\DeclareMathOperator{\gldim}{gldim}

\DeclareMathOperator{\gr}{gr}

\DeclareMathOperator{\mcm}{mcm}

\DeclareMathOperator{\uExt}{{\underline{Ext}}}
\DeclareMathOperator{\uHom}{{\underline{Hom}}}

\DeclareMathOperator{\id}{injdim}

\def\kk{\mathbbm{k}}

\def\gr{\operatorname{gr}}
\def\tor{\operatorname{tors}}
\def\Gr{\operatorname{Gr}}

\def\qgr{\operatorname{qgr}}

\def\11{\mathbf{1}}

\begin{document}

\title[Generalized Kn\"{o}rrer's Periodicity Theorem]
{Generalized Kn\"{o}rrer's Periodicity Theorem}

\author{Ji-Wei He, Xin-Chao Ma and Yu Ye}

\address{He: Department of Mathematics,
Hangzhou Normal University,
Hangzhou Zhejiang 311121, China}
\email{jwhe@hznu.edu.cn}

\address{Ma: School of Mathematical Sciences, University of Science and Technology of China, Hefei Anhui 230026, China}
\email{mic@mail.ustc.edu.cn}

\address{Ye: School of Mathematical Sciences, University of Science and Technology of China, Hefei Anhui 230026, China\hfill \break
\indent\qquad CAS Wu Wen-Tsun Key Laboratory of Mathematics, University of Science and Technology of China, Hefei, Anhui, 230026, PR China}
\email{yeyu@ustc.edu.cn}

\begin{abstract}
Let $A$ be a noetherian Koszul Artin-Schelter regular algebra, and let $f\in A_2$ be a central regular element of $A$. The quotient algebra $A/(f)$ is usually called a (noncommutative) quadric hypersurface. In this paper, we use the Clifford deformation to study the quadric hypersurfaces obtained from the tensor products. We introduce a notion of simple graded isolated singularity and proved that, if $B/(g)$ is a simple graded isolated singularity of 0-type, then there is an equivalence of triangulated categories $\underline{\mcm}A/(f)\cong\underline{\mcm}(A\otimes B)/(f+g)$ of the stable categories of maximal Cohen-Macaulay modules. This result may be viewed as a generalization of Kn\"{o}rrer's periodicity theorem. As an application, we study the double branch cover $(A/(f))^\#=A[x]/(f+x^2)$ of a noncommutative conic $A/(f)$.
\end{abstract}





\maketitle


\setcounter{section}{-1}
\section{Introduction}
In noncommutative projective geometry, Artin-Schelter regular algebras are usually regarded as the coordinate rings of noncommutative projective spaces. Let $A$ be a noetherian Koszul Artin-Schelter regular algebra, and let $f\in A_2$ be a central regular element of $A$. The quotient algebra $A/(f)$ is usually called a {\it noncommutative quadric hypersurface}. Noncommutative quadric hypersurfaces have got lots of attentions in recent years (see  \cite{SvdB,CKMW,MU1,MU,HU,HMM,Ue,Ue2}, etc). To study the graded Cohen-Macaulay modules of $A/(f)$, Smith and Van den Bergh introduced in \cite{SvdB} a finite dimensional algebra $C(A/(f))$, and proved that there is an equivalence of triangulated categories
$$\underline{\mcm} A/(f)\cong D^b(\text{mod} C(A/(f))),$$
where $\underline{\mcm}A/(f)$ is the stable category of graded maximal Cohen-Macaulay modules over $A/(f)$ and $D^b(\text{mod}C(A/(f)))$ is the bounded derived category of finite dimensional modules over $C(A/(f))$. It is proved in \cite{HMM} that $C(A/(f))$ is a Morita invariant. Hence, in some sense, the Cohen-Macaulay representations of $A/(f)$ are determined by $C(A/(f))$.

However, one cannot recover the properties of $A/(f)$ completely from $C(A/(f))$. For instance, let $A=\mathbb C[x,y]$, $A'=\mathbb C_{-1}[x,y]$ and $f=x^2+y^2$. Then $C(A/(f))\cong C(A'/(f))$ and $\underline{\mcm}A/(f)\cong \underline{\mcm}A'/(f)\cong D^b(\text{mod}\mathbb C^{2})$. But if we consider the double branched covers of $A/(f)$ and $A'/(f)$ respectively, there will be a different situation. Indeed, $\underline{\mcm} (A/(f))^\#\cong D^b(\text{mod}\mathbb C)$ and $\underline{\mcm}(A'/(f))^\#\cong D^b(\text{mod}\mathbb C^4)$.

In \cite{HY}, we introduced the notion of {\it Clifford deformation} of a Koszul Frobenius algebra. Associated to every noncommutative quadric hypersurface $A/(f)$, there is a Clifford deformation $C_{A^!}(\theta_f)$, which is a strongly $\mathbb Z_2$-graded algebra. It turns out that the degree zero part of $C_{A^!}(\theta_f)$ is isomorphic to the finite dimensional algebra $C(A/(f))$ as introduced in \cite{SvdB}. The $\mathbb Z_2$-graded algebra $C_{A^!}(\theta_f)$ may recover enough information of $A/(f)$. For instance, the reason that $\underline{\mcm} (A/(f))^\#$ differs from $\underline{\mcm} (A'/(f))^\#$ in the previous paragraph, is because the $\mathbb Z_2$-graded algebra $C_{A^!}(\theta_f)$ associated to $A/(f)$ is a simple graded  algebra, while $C_{{A'}^!}(\theta_f)$ is not a simple graded algebra (see  Examples \ref{ex-s1}, \ref{ex-s2}).

In this paper, we use the Clifford deformation to study the quadric hypersurfaces obtained from the tensor products of Koszul Artin-Schelter regular algebras. Let $A$ and $B$ be Koszul Artin-Schelter regular algebras, and let $f\in A_2,g\in B_2$ be central regular elements. Suppose that $A\otimes B$ is noetherian. We prove that $(A\otimes B)/(f+g)$ is a graded isolated singularity provided both $A/(f)$ and $B/(g)$ are graded isolated singularities (see  Theorem \ref{thm-isolat}). We introduce
a notion of a {\it simple graded isolated singularity} in Section \ref{sec2} and then prove a generalized version of Kn\"{o}rrer's Periodicity Theorem (see  Theorem \ref{thm-k}). In particular, it will recover the classical Kn\"{o}rrer's Periodicity Theorem of quadric singularities (see  Remark \ref{rmk}). As an application, we study the double branch cover $(A/(f))^{\#}$ of a noncommutative conic $A/(f)$ classified in \cite{HMM} and proved that, in noncommutative case, $\underline{\mcm}(A/(f))^{\#}\cong \underline{\mcm}(A/(f))\times \underline{\mcm}(A/(f))$ (see  Corollary \ref{cor}).

We assume the ground field $\kk=\mathbb C$, and all the vector spaces and algebras are over $\mathbb C$.

\section{Morita theory of $\mathbb Z_2$-graded algebras revisited}
In this section, we will recall some Morita type properties of $\mathbb Z_2$-graded algebras.

Let $E=E_0\oplus E_1$ be a finite dimensional $\mathbb Z_2$-graded algebras.
Denote $\gr_{\mathbb Z_2} E$ to be the category whose objects are finite dimensional graded right $E$-modules, and whose hom-sets are denoted by $\Hom_{\gr_{\mathbb Z_2} E}(M,N)$ consisting of right $E$-module homomorphisms which preserve the gradings.  We use $\mathrm{mod}E$ to denote the category whose objects are all the finite dimensional right $E$-modules (ignoring the grading of $E$), and whose hom-sets are denoted by $\Hom_E(M,N)$ consisting of all the right $E$-module homomorphisms. Note that if $M$ is finite dimensional, then $\Hom_E(M,N)$ is a $\mathbb Z_2$-graded vector space.

We say that $E$ is {\it graded semisimple} if $E$ is a direct sum of simple objects in $\gr_{\mathbb Z_2} E$. Let $J^g(M)$ be the graded Jacobson radical of $M$, which is the intersection of all graded maximal submodules of $M$, it is clear that $J^g({}_EE)=J^g(E_E)$ is a two-sided nilpotent graded ideal and $E$ is graded semisimple if and only if $J^g(E)=0$.

A $\mathbb Z_2$-graded algebra $E$ is called a {\it graded division algebra} if each non-zero homogeneous element of $E$ is invertible.

Let $F$ be another finite dimensional $\mathbb Z_2$-graded algebra. Denote $E\hat\otimes F$ to be the twisted tensor product of $E$ and $F$. The multiplications of elements in $E\hat\otimes F$ is defined by $$(x\hat\otimes y)(x'\hat\otimes y')=(-1)^{|y||x'|}xx'\hat\otimes yy',$$  where $x,x'\in E$ and $y,y'\in F$ are homogeneous elements, and $|x'|$ denotes the degree of $x'$. Note that $E\hat\otimes F$ is also a $\mathbb Z_2$-graded algebra.

We say a $\mathbb Z_2$-graded algebra $E$ is {\it graded Morita equivalent} to $F$ if there is a finitely generated $\mathbb Z_2$-graded bimodule ${}_FP_E$ such that $\Hom_E(P,-):\gr_{\mathbb Z_2} E\longrightarrow \gr_{\mathbb Z_2} F$ is an equivalence of abelian categories.

The following results are well known (see  \cite[Lemmas 3.4, 3.10]{Z} and \cite[Theorem 2.10.10]{NV}).
\begin{lemma}\label{lem-z2} Let $E$ and $F$ be $\mathbb Z_2$-graded algebras.
\begin{itemize}
\item [(i)]Assume that $E$ and $F$ are graded Morita equivalent to $E'$ and $F'$ respectively. Then $E\hat\otimes F$ is graded Morita equivalent to $E'\hat\otimes F'$.
\item [(ii)]{\rm($\mathbb Z_2$-graded version of Wedderburn-Artin Theorem)} If $E$ is $\mathbb Z_2$-graded semisimple, then $E$ is isomorphic, as a graded algebra, to a direct product of finitely many matrix algebras over some division algebras.
\end{itemize}
\end{lemma}

Since the ground field is $\mathbb C$, there are only two classes of finite dimensional graded division algebras. Let $ \mathbb G=\{1,\sigma\}$ be a group of order 2, and let $\mathbb{CG}  $ be the group algebra. Then $\mathbb {CG}$ is a $\mathbb Z_2$-graded algebra by setting $|\sigma|=1$ and $|1|=0$.

\begin{lemma}\label{lem-div} Let $E$ be a finite dimensional $\mathbb Z_2$-graded division algebra over $\mathbb C$. Then $E$ is isomorphic to either $\mathbb {CG}$ or $\mathbb C$, where $\mathbb C$ is viewed as a $\mathbb Z_2$-graded algebra concentrated in degree $0$.
\end{lemma}

\begin{proof}
Since $E_0$ is a finite dimensional division algebra over $\mathbb{C}$, hence $E_0=\mathbb{C}$. If $E_1\ne 0$, for any non-zero elements $x,y\in E_1$, we have $x^{-1}y\in \mathbb{C}$ and $x=\lambda y$ for some $\lambda\in\mathbb C$. Therefore, $\dim E_1=1$ and $E$ is isomorphic to the skew group algebra $ \mathbb{C}\# \mathbb G$. The action of $\sigma\in \mathbb G$ is $\pm 1$, and both of them are isomorphic to the group algebra $\mathbb{CG}$.
\end{proof}

For a $\mathbb Z_2$-graded algebra $E$, we write $\gldim_{\mathbb Z_2} E$ for the graded right global dimension of $E$.

\begin{lemma}\cite[Corollary 7.4]{HY}\label{lem-gl} Let $E$ be a $\mathbb Z_2$-graded algebra. Then
$$\gldim_{\mathbb Z_2} E=\gldim_{\mathbb Z_2} E\hat\otimes \mathbb C\mathbb G.$$
\end{lemma}

\begin{proposition}\label{prop-gl} Let $E$ be a $\mathbb Z_2$-graded algebra over $\mathbb C$, and let $F$ be a  $\mathbb Z_2$-graded semi-simple algebra over $\mathbb C$. Then $\gldim_{\mathbb Z_2} E\hat\otimes F=\gldim_{\mathbb Z_2}E$. In particular, $E\hat\otimes F$ is graded semisimple if and only if $E$ is graded semisimple.
\end{proposition}
\begin{proof} By the $\mathbb Z_2$-graded version of Wedderburn-Artin Theorem, $F$ is isomorphic to a direct product of finitely many matrix algebras over some graded division algebras. Note that a matrix algebra over a graded division algebra (with possible degree shifts) is graded Morita equivalent to a graded division algebra, and hence a matrix algebra over a graded division algebra is graded Morita equivalent to either $\mathbb {CG}$ or $\mathbb C$ by Lemma \ref{lem-div}. Then the result follows from Lemma \ref{lem-z2}(i) and Lemma \ref{lem-gl}.
\end{proof}

We have the following Morita cancellation type property (see  \cite{LWZ}).

\begin{proposition}\label{prop-can} Let $E$ and $F$ be $\mathbb Z_2$-graded algebras. Then $E\hat\otimes\mathbb {CG}$ is graded Morita equivalent to $F\hat\otimes \mathbb {CG}$ if and only if $E$ is graded Morita equivalent to $F$.
\end{proposition}
\begin{proof} Suppose that $E\hat\otimes\mathbb {CG}$ is graded Morita equivalent to $F\hat\otimes \mathbb {CG}$. By Lemma \ref{lem-z2}(i), $E\hat\otimes\mathbb {CG}\hat\otimes\mathbb{CG}$ is graded Morita equivalent to $F\hat\otimes \mathbb {CG}\hat\otimes\mathbb {CG}$. By \cite[Lemma 7.2]{HY}, the $\mathbb Z_2$-graded algebra $E\hat\otimes\mathbb {CG}\hat\otimes\mathbb{CG}$ is graded Morita equivalent to $E$, and $F\hat\otimes\mathbb {CG}\hat\otimes\mathbb{CG}$ is graded Morita equivalent to $F$. Hence $E$ is graded Morita equivalent to $F$. The other direction is a consequence of Lemma \ref{lem-z2}(i).
\end{proof}

We next focus on strongly $\mathbb Z_2$-graded algebras. Recall that a $\mathbb Z_2$-graded algebra $E$ is said to be {\it strongly graded} if $E_1E_1=E_0$. Let $E$ and $F$ be strongly $\mathbb Z_2$-graded algebras. Then $E\hat\otimes F$ is also a strongly $\mathbb Z_2$-graded algebra.
\begin{proposition}\label{prop-str} Let $E$ be a finite dimensional strongly $\mathbb Z_2$-graded algebra. Then there is an equivalence of abelian categories $$\gr_{\mathbb Z_2} E\hat\otimes\mathbb{CG}\cong\text{\rm mod} E,$$ where $\text{\rm mod} E$ is the category of finite dimensional right $E$-modules (ignoring the grading of $E$).
\end{proposition}
\begin{proof} As a vector space, the degree 0 part of $E\hat\otimes\mathbb{CG}$ is equal to
$$E_1\otimes \mathbb{C}\sigma\oplus E_0\otimes \mathbb C=E_1\otimes \mathbb{C}\sigma\oplus E_0,$$ and the multiplication of $(E\hat\otimes\mathbb{CG})_0$ is given as following (we temporarily write the multiplication of $E\hat\otimes \mathbb {CG}$ by the symbol ``$*$''): for $a,b\in E_1$, $c,d\in E_0$, we have $$(a\otimes \sigma)*(b\otimes \sigma)=-ab,\ \ (a\otimes \sigma)*c=ac\otimes\sigma,$$ $$c*(a\otimes \sigma)=ca\otimes\sigma,\ \ c*d=cd.$$ We temporarily write $E^\natural$ for the ungraded algebra obtained from $E$ by ignoring the grading. Define a linear map $$\Xi:(E\hat\otimes\mathbb{CG})_0\longrightarrow E^\natural,$$ by setting $\Xi(a\otimes \sigma)=\sqrt{-1}a$ if $a\in E_1$, and $\Xi(a)=a$ if $a\in E_0$. Then it is straightforward to check that $\Xi$ is an isomorphism of algebras.

Since $E$ is strongly graded, $E\hat\otimes\mathbb {CG}$ is also strongly graded and there is an equivalence of abelian categories $\gr_{\mathbb Z_2}E\hat\otimes\mathbb{CG}\cong \text{mod} (E\hat\otimes\mathbb{CG})_0$, which is in turn equivalent to $\text{mod} E^\natural$.
\end{proof}

\section{Products of quadric hypersurfaces} \label{sec2}

Let us recall some notations and terminologies. An $\mathbb N$-graded algebra $A=\oplus_{n\in\mathbb N}A_n$ is called a {\it connected graded algebra} if $A_0=\mathbb C$. A connected graded algebra is said to be {\it locally finite}, if $\dim A_n<\infty$ for all $n$. Let $\Gr A$ denote the category whose objects are graded right $A$-modules, and whose morphisms are right $A$-module morphisms which preserve the gradings of modules. For a graded right $A$-module $X$ and an integer $l$, we write $X(l)$ for the graded right $A$-module whose $i$th component is $X(l)_i=X_{i+l}$.

A locally finite connected graded algebra $A$ is called a {\it Koszul algebra} (see  \cite{P}) if the trivial module $\kk_A$ has a {\it linear free resolution}; i.e., $$\cdots\longrightarrow P^{n}\longrightarrow\cdots\longrightarrow P^1\longrightarrow P^0\longrightarrow\mathbb C\longrightarrow0,$$ where $P^n$ is a graded free module generated in degree $n$ for each $n\ge0$. Note that a Koszul algebra is a {\it quadratic algebra}, that is, $A\cong T(V)/(R)$, where $V$ is a finite dimensional vector space and $R\subseteq V\otimes V$. If $A$ is a Koszul algebra, the {\it quadratic dual} of $A$ is the quadratic algebra $A^!=T(V^*)/(R^\bot)$, where $V^*$ is the dual vector space and $R^\bot\subseteq V^*\otimes V^*$ is the orthogonal complement of $R$.

For graded right $A$-modules $X$ and $Y$, denote $\uHom_A(X,Y)=\bigoplus_{i\in\mathbb Z}\Hom_{\Gr A}(X,Y(i))$. Then $\uHom_A(X,Y)$ is a $\mathbb Z$-graded vector space. We write $\uExt_A^i$ for the $i$th derived functor of $\uHom_A$. Note that $\uExt_A^i(X,Y)$ is also a $\mathbb Z$-graded vector space for each $i\ge0$.

\begin{definition} \cite{AS} A connected graded algebra $A$ is called an {\it Artin-Schelter Gorenstein algebra} of injective dimension $d$ if
\begin{itemize}
  \item [(i)] $A$ has finite injective dimension $\id {}_AA=\id A_A=d<\infty$,
  \item [(ii)] $\uExt_A^i(\mathbb{C}_A,A_A)=0$ for $i\neq d$, and $\uExt_A^d(\mathbb{C}_A,A_A)\cong {}_A\mathbb{C}(l)$,
  \item [(iii)] the left version of (ii) holds.
\end{itemize}
If further, $A$ has finite global dimension, then $A$ is called an {\it Artin-Schelter regular} algebra.
\end{definition}

An Artin-Schelter regular algebra $A$ is called a {\it quantum polynomial algebra} if
(i) $A$ is a noetherian domain, (ii) $A$ is a Koszul algebra with Hilbert series $$H_t(A):=\sum_{i\ge0} t^n\dim (A_n)=\frac{1}{(1-t)^d}$$ for some $d>0$.

Let $A$ be a noetherian Artin-Schelter Gorenstein algebra. Let $X$ be a graded right $A$-module, and let $\Gamma(X)=\{x\in X|xA\text{ if finite dimensional}\}$. We obtain a functor $\Gamma:\Gr A\longrightarrow\Gr A$. We write $R^i\Gamma$ for the $i$th right derived functor of $\Gamma$.
Assume that $\id {}_AA=\id A_A=d$. A finitely generated graded right $A$-module $X$ is called a {\it maximal Cohen-Macaulay module} if $R^i\Gamma(X)=0$ for all $i\neq d$. Let $\mcm A$ be the subcategory of $\Gr A$ consisting of all the maximal Cohen-Macaulay modules. Let $\underline{\mcm} A$ be the stable category of $\mcm A$. Then $\underline{\mcm} A$ is a triangulated category.

Let $V$ and $U$ be finite dimensional vector spaces, and let $A=T(V)/(R_A)$ and $B=T(U)/(R_B)$ be Koszul algebras. Let us consider the tensor algebra $A\otimes B$.  We view $A$ and $B$ as graded subalgebras of $A\otimes B$ through the injective maps $A\hookrightarrow A\otimes B, a\mapsto a\otimes 1$ and $B\hookrightarrow A\otimes B, b\mapsto 1\otimes b$ respectively. Also, we identify $A_1$ with $V$ and $B_1$ with $U$. So, the generating space of $A\otimes B$ is $W=V\oplus U$, and $A\otimes B\cong T(W)/(R_{A\otimes B})$, where
\begin{equation}\label{eqr}
  R_{A\otimes B}=R_A\oplus [V,U]\oplus R_B,
\end{equation}
in which (see  \cite[Chapter 3]{M}) $$[V,U]=\{v\otimes u-u\otimes v|v\in V,u\in U\}.$$ Moreover, $A\otimes B$ is also a Koszul algebra.
Let $A^!$ and $B^!$ be the quadratic dual algebras of $A$ and $B$ respectively. The quadratic dual algebra of $A\otimes B$ is the following graded algebra \cite[3.5]{M} $$(A\otimes B)^!=A^!\hat\otimes B^!.$$

Now we assume both $A$ and $B$ are Koszul Artin-Schelter regular algebras. Then $A^!$ and $B^!$ are graded Frobenius (see  \cite[Proposition 5.1]{Sm}). It is easy to see that the graded tensor product $A^!\hat\otimes B^!$ is also a graded Frobenius algebra. Hence by \cite[Proposition 5.1]{Sm} again, $A\otimes B$ is an Artin-Schelter regular algebra.

In this section, we are interested in the quadric hypersurfaces obtained from $A\otimes B$.

Now assume both $A$ and $B$ are noetherian Koszul Artin-Schelter regular algebras, further more, $A\otimes B$ is also noetherian. The following lemma is clear.

\begin{lemma}
Let $f\in A_2$ and $g\in B_2$ be central regular elements of $A$ and $B$ respectively. View $f$ and $g$ as elements in $A\otimes B$. Then $h=f+g$ is a central regular element of $A\otimes B$.
\end{lemma}

By the Rees Lemma \cite[Proposition 3.4(b)]{Le}, $(A\otimes B)/(h)$ is a noetherian Koszul Artin-Schelter Gorenstein algebra. We are going to consider the Cohen-Macaulay modules over $(A\otimes B)/(h)$. Let us recall some notations from \cite{SvdB,HY}.

Let $E=T(X)/(R_E)$ be a Koszul Frobenius algebra. A linear map $\theta:R_E\to \mathbb{C}$ is called a {\it Clifford map} (see  \cite[Definition 2.1]{HY}) if $$(\theta\otimes 1-1\otimes \theta)(X\otimes R_E\cap R_E\otimes X)=0.$$
Given a Clifford map $\theta:R_E\to\mathbb{C}$, define an associative algebra
\begin{equation}\label{eq-cliff}
  C_E(\theta)=T(X)/(r-\theta(r):r\in R).
\end{equation}
The algebra $C_E(\theta)$ is called the {\it Clifford deformation} of $E$ associated to $\theta$. We may view $T(X)$ as a $\mathbb Z_2$-graded algebra by taking $T(X)_0=\mathbb{C}\oplus (\bigoplus_{n\ge1}X^{\otimes 2n})$ and $T(X)_1=\bigoplus_{n\ge1}X^{\otimes 2n-1}$. Since the definition relations of $C_E(\theta)$ lie in degree 0 component of $T(X)$, we may view $C_E(\theta)$ as a $\mathbb Z_2$-graded algebra.

Now let $S=T(Y)/(R_S)$ be a noetherian Koszul Artin-Schelter regular algebra. Denote by $\pi:T(Y)\to S$ the natural projection map. Let $z\in S_2$ be a central regular element of $S$. Pick an element $r_0\in Y\otimes Y$ such that $\pi(r_0)=z$. Let $S^!=T(Y^*)/(R_S^\bot)$ be the quadratic dual algebra of $S$. Define a linear map
\begin{equation}\label{eqtheta}
  \theta_z\colon R_S^\bot\to\mathbb{C},\text{ by setting }\theta_z(\alpha)=\alpha(r_0),\ \forall \alpha\in R_S^\bot.
\end{equation}
Note that the map $\theta_z$ is independent of the choice of $r_0$.
The following results were proved in \cite{HY,SvdB}.

\begin{theorem}\label{thm-HY} Retain the notions as above.
\begin{itemize}
  \item [(i)] $\theta_z$ is a Clifford map of $S^!$;
  \item [(ii)] The Clifford deformation $C_{S^!}(\theta_z)$ is a finite dimensional strongly $\mathbb Z_2$-graded Frobenius algebra;
  \item [(iii)] There is an equivalence of triangulated categories $$\underline{\mcm} S/(z)\cong D^b(\gr_{\mathbb Z_2} C_{S^!}(\theta_z))\cong D^b(\mathrm{mod} C_{S^!}(\theta_z)_0) ,$$ where $\underline{\mcm} S/(z)$ is the stable category of graded maximal Cohen-Macaulay modules of $S/(z)$, and $\gr_{\mathbb Z_2} C_{S^!}(\theta_z)$ is the category of finite dimensional right $\mathbb Z_2$-graded modules of $C_{S^!}(\theta_z)$.
\end{itemize}
\end{theorem}

Now let us go back to the quadric hypersurface $(A\otimes B)/(h)$. Recall $f\in A_2$, $g\in B_2$ and $h=f+g$, and $A=T(V)/(R_A)$ and $B=T(U)/(R_B)$. Then $A^!=T(V^*)/(R^\bot_A)$, $B^!=T(U^*)/(R^\bot_B)$ and $(A\otimes B)^!=T(V^*\oplus U^*)/(R^\bot_{A\otimes B})$. By Theorem \ref{thm-HY}, we have Clifford maps
\begin{equation}\label{eq-map1}
  \theta_f:R_A^\bot\to\mathbb{C},
\end{equation}
\begin{equation}\label{eq-map1}
  \theta_g:R_B^\bot\to\mathbb{C},
\end{equation}
\begin{equation}\label{eq-map1}
  \theta_h:R_{A\otimes B}^\bot\to\mathbb{C}.
\end{equation}

\begin{lemma}\label{lem-clifford}
Retain the notations as above. we have an isomorphism of $\mathbb Z_2$-graded algebras $$C_{(A\otimes B)^!}(\theta_h)\cong C_{A^!}(\theta_f)\hat\otimes C_{B^!}(\theta_g).$$
\end{lemma}
\begin{proof} Let $\pi_A:T(V)\to A$, $\pi_B:T(U)\to B$ and $\pi_{A\otimes B}:T(V\oplus U)\to A\otimes B$ be the natural projection maps. Pick elements $r_1\in V\otimes V$ and $r_2\in U\otimes U$ such that $\pi_A(r_1)=f$ and $\pi_B(r_2)=g$. Then $\pi_{A\otimes B}(r_1+r_2)=f+g=h$. Let us compute the generating relations of $C_{(A\otimes B)^!}(\theta_h)$. By (\ref{eq-cliff}), $C_{(A\otimes B)^!}(\theta_h)=T(V^*\oplus U^*)/I$, where $I$ is generated by the space $$\mathcal{R}:=\{\alpha-\theta_h(\alpha)|\alpha\in R^\bot_{A\otimes B}\}\subseteq (V^*\oplus U^*)\otimes(V^*\oplus U^*).$$ Note that $R_{A\otimes B}=R_A\oplus [V,U]\oplus R_B$. It follows $$R^\bot_{A\otimes B}=R^\bot_A\oplus [V^*,U^*]_+\oplus R_B^\bot,$$ where $R^\bot_A\subseteq V^*\otimes V^*$, $R^\bot_B\subseteq U^*\otimes U^*$ and $[V^*,U^*]_+=\{\alpha\otimes \beta+\beta\otimes \alpha|\alpha\in V^*,\beta\in U^*\}$. Hence
\begin{equation}\label{eq-r}
  \mathcal{R}=\{\alpha-\alpha(f)|\alpha\in R^\bot_A\}+\{\beta-\beta(g)|\beta\in R^\bot_B\}+[V^*,U^*]_+.
\end{equation}
On the other hand, $C_{A^!}(\theta_f)\hat\otimes C_{B^!}(\theta_g)=T(V^*\oplus U^*)/J$, where $J$ is generated by the space $$\mathcal{\hat R}:=\{\alpha-\theta_f(\alpha)|\alpha\in R^\bot_A\}+\{\beta-\theta_g(\beta)|\beta\in R^\bot_B\}+[V^*,U^*]_+.$$ Since $\theta_f(\alpha)=\alpha(f)$ and $\theta_g(\beta)=\theta(g)$ by (\ref{eqtheta}), it follows that $\mathcal{R}=\mathcal{\hat R}$. Hence
$$C_{(A\otimes B)^!}(\theta_h)\cong C_{A^!}(\theta_f)\hat\otimes C_{B^!}(\theta_g).$$
\end{proof}

Let $A$ be a noetherian graded algebra. Denote $\gr A$ for the category of finitely generated graded right $A$-modules, and $\tor A$ for the category of finite dimensional graded right $A$-modules. Let $\qgr A=\gr A/\tor A$. Recall that $A$ is called a {\it graded isolated singularity} if $\qgr A$ has finite global dimension. Lemma \ref{lem-clifford} implies the following result.

\begin{theorem}\label{thm-isolat} Let $A$ and $B$ be noetherian Koszul Artin-Schelter regular algebras. Suppose that $A\otimes B$ is also noetherian. Assume that $f\in A$ and $g\in B$ are central regular homogeneous elements of degree 2. Suppose that $B$ is a graded isolated singularity. Then $(A\otimes B)/(f+g)$ is a graded isolated singularity if and only if $A/(f)$ is.
\end{theorem}
\begin{proof} Let $h=f+g$. By Lemma \ref{lem-clifford}, we have an isomorphism of $\mathbb Z_2$-graded algebras
$$C_{(A\otimes B)^!}(\theta_h)\cong C_{A^!}(\theta_f)\hat\otimes C_{B^!}(\theta_g).$$
By assumption, $B/(f)$ is a graded isolated singularity. Hence, by \cite[Theorem 6.3]{HY}, $C_{B^!}(\theta_g)$ is a $\mathbb Z_2$-graded semisimple algebra.

By Proposition \ref{prop-gl}, $C_{(A\otimes B)^!}(\theta_h)$, which is isomorphic to $C_{A^!}(\theta_f)\hat\otimes C_{B^!}(\theta_g)$, is a $\mathbb Z_2$-graded semisimple algebra if and only if $C_{A^!}(\theta_f)$ is a $\mathbb Z_2$-graded semisimple algebra. Applying \cite[Theorem 6.3]{HY} again, we obtain that $(A\otimes B)/(f)$ is a graded isolated singularity if and only if $A/(f)$ is.
\end{proof}

\begin{definition} Let $A$ be a noetherian Koszul Artin-Schelter regular algebra, and let $f\in A_2$ be a central regular element. If the $\mathbb Z_2$-graded algebra $C_{A^!}(\theta_f)$ is a simple $\mathbb Z_2$-graded algebra, then we call $A/(f)$ is a {\it simple graded isolated singularity}.
\end{definition}

Since the ground field is $\mathbb C$, there are two classes of simple graded algebras: matrix algebras over the $\mathbb Z_2$-graded algebra $\mathbb {CG}$, and matrix algebras over the $\mathbb Z_2$-graded algebra $\mathbb C$. In the above definition, if $C_{A^!}(\theta_f)$ is a matrix algebra over the $\mathbb Z_2$-graded $\mathbb C$, then we further call $A/(f)$ is a {\it simple graded isolated singularity of 0-type}; if $C_{A^!}(\theta_f)$ is a matrix algebra over the $\mathbb Z_2$-graded $\mathbb {CG}$, then we further call $A/(f)$ is a {\it simple graded isolated singularity of 1-type}. It is not hard to see that, if both $A/(f)$ and  $B/(g)$ are simple graded isolated singularities of 1-type and $A\otimes B$ is also noetherian, then $(A\otimes B)/(f+g)$ is a simple graded isolated singularity of 0-type.

\begin{example}\label{ex-s1} Let $A=\mathbb C[x,y]$, and $f=x^2+y^2$. Then
$$C_{A^!}(\theta_f)\cong \mathbb C_{-1}[u,v]/(u^2-1,v^2-1)\cong \mathbb {C}G\hat\otimes \mathbb {C}G\cong \mathbb M_2(\mathbb C),$$
where $\mathbb M_2(\mathbb C)$ is the $2\times2$-matrix algebra which is viewed as a $\mathbb Z_2$-graded algebra by setting
$$\mathbb M_2(\mathbb C)_0=\left[
                                                                    \begin{array}{cc}
                                                                      \mathbb C & 0 \\
                                                                      0 & \mathbb C\\
                                                                    \end{array}
                                                                  \right],\quad \mathbb M_2(\mathbb C)_1=\left[
                                                                    \begin{array}{cc}
                                                                      0&\mathbb C  \\
                                                                      \mathbb C&0\\
                                                                    \end{array}
                                                                  \right].$$
Hence $\mathbb C[x,y]/(x^2+y^2)$ is a simple graded isolated singularity of 0-type.
\end{example}
\begin{example} \label{ex-s3} Let $A=\mathbb C[x]$, and $f=x^2$. Then
$$C_{A^!}(\theta_f)\cong \mathbb{C}[x]/(x^2-1)\cong\mathbb {C}G.$$
Therefore $A/(f)$ is a simple graded isolated singularity of 1-type.
\end{example}
\begin{example} \label{ex-s2} Let $A=\mathbb C_{-1}[x,y]$, and $f=x^2+y^2$. Then $$C_{A^!}(\theta_f)\cong \mathbb C[u,v]/(u^2-1,v^2-1)\cong \mathbb {CG}\times \mathbb {CG}.$$ Hence $A/(f)$ is not a simple graded isolated singularity.
\end{example}

\begin{lemma}\label{simp}
If $C_{A^{!}}(\theta_f)_0$ is a simple algebra, then $A/(f)$ is a simple graded isolated singularity of 1-type.
\end{lemma}
\begin{proof}
Since $C_{A^{!}}(\theta_f)_0$ is simple and $C_{A^{!}}(\theta_f)$ is strongly graded, we obtain that $C_{A^{!}}(\theta_f)$ is simple graded. If $C_{A^{!}}(\theta_f)$ is a matrix over $\mathbb{C}$, it must concentrate in degree 0 since $C_{A^{!}}(\theta_f)_0$ is simple. But in this case, $C_{A^{!}}(\theta_f)$ can not be strongly graded. Therfore, $C_{A^{!}}(\theta_f)$ must be a matrix over $\mathbb{CG}$.
\end{proof}

The next is an example of noncommutative simple graded isolated singularities.
\begin{example}  Let $A=\mathbb C\langle x_1,\dots,x_5\rangle/(r_1,\dots,r_{10})$, where the generating relations are as follows:
\begin{center}
  $r_1=x_1x_2-x_2x_1, r_2=x_1x_3+x_3x_1, r_3=x_1x_4+x_4x_1,$\\
   $r_4=x_1x_5+x_5x_1,r_5=x_2x_3-x_3x_2, r_6=x_2x_4+x_4x_1,$ \\
   $r_7=x_2x_5+x_5x_2, r_8=x_3x_4-x_4x_3,r_9=x_3x_5+x_5x_3,r_{10}=x_4x_5+x_5x_4$.
\end{center}
Let $f=x_1^2+\cdots+x_5^2$. By \cite[Section 5.4.2]{MU}, $\underline{\mcm}A/(f)\cong D^b(\mathbb C)$. Hence $C_{A^!}(\theta_f)_0$ is simple and $A/(f)$ is a simple graded isolated singularity of 1-type by Lemma \ref{simp}.
\end{example}

\begin{lemma}
Let $A$ be a quantum polynomial algebra of global dimension $n$, and let $f\in A_2$ be a central regular element, then $\dim C_{A^!}(\theta_f)=2^n$. Moreover

{\rm(i)} If $A/(f)$ is a simple graded isolated singularity of 0-type, then $n$ is even.

{\rm(ii)} If  $A/(f)$ is a simple graded isolated singularity of 1-type, then $n$ is odd.
\end{lemma}

\begin{proof}
Since $A$ has Hilbert series $1/(1-t)^n$, the Hilbert series of $A^!$ is $(1+t)^n$. Therefore, $\dim A^!=2^n$. Note that a Clifford deformation dose not change the dimension, we have $\dim C_{A^!}(\theta_f)=2^n$. If $A/(f)$ is a simple graded isolated singularity of 0-type, then $C_{A^!}(\theta_f)$ is a matrix over $\mathbb{C}$ and $\dim C_{A^!}(\theta_f)$ is a square of an integer, so $n$ must be even. If $A/(f)$ is a simple graded isolated singularity of 1-type, then $C_{A^!}(\theta_f)$ is a matrix over $\mathbb{CG}$, so $n$ must be odd.
\end{proof}

\begin{remark}\label{rmk-2}
We don't know by now when $A/(f)$ is a simple isolated singularity. It seems that the \textit{rank} of $f$ introduced by Mori-Ueyama in \cite{MU} is a candidate tool. Let $A$ be a quantum polynomial algebra of global dimension $n$, and let $f\in A_2$ be a central regular element. The \textit{rank} of $f$ (see  \cite[Definition 5.5]{MU}) is defined by
$$\mathrm{rank} f=\min\Big\{r\in\mathbb{N}^{+}\mid f=\sum_{i=1}^ru_iv_i,~u_i,v_i\in A_1\Big\}.$$
It has been proved that $C_{A^{!}}(\theta_f)_0$ has no non-zero modules of dimension less than $\mathrm{rank} f$ (see  \cite[Lemma 5.10]{MU}). Since $\dim C_{A^!}(\theta_f)_0=2^{n-1}$, if $n$ is odd and $\mathrm{rank}f\ge 2^{\frac{n-1}{2}}$, then $C_{A^!}(\theta_f)_0$ is a matrix over $\mathbb{C}$ and $A/(f)$ is a simple isolated singularity of 1-type by Lemma \ref{simp}. For example, we can choose $A=\mathbb{C}[x,y,z]$ and $f=x^2+y^2+z^2$, then $n=3$ and $\mathrm{rank}f=2$. In this case, $C_{A^!}(\theta_f)_0=M_2(\mathbb{C})$ and $C_{A^!}(\theta_f)=M_2(\mathbb{CG})$.
\end{remark}
\begin{proposition} Let $A$ be a noetherian Koszul Artin-Schelter regular algebra, and let $f\in A_2$ be a central regular element.
\begin{itemize}
\item [(i)] If $A/(f)$ is a simple graded isolated singularity of 0-type, then $A/(f)$ has two indecomposable nonprojective graded Cohen-Macaulay modules (up to isomorphisms and degree shifts);
\item [(ii)] If $A/(f)$ is a simple graded isolated singularity of 1-type, then $A/(f)$ has one indecomposable nonprojective graded Cohen-Macaulay module (up to isomorphisms and degree shifts).
\end{itemize}
\end{proposition}
\begin{proof} (i) Since $A/(f)$ is a simple graded isolated singularity of 0-type, then $C_{A^!}(\theta_f)$ is a matrix algebra over the $\mathbb Z_2$-graded algebra $\mathbb C$. Then $C_{A^!}(\theta_f)\cong \End_{\mathbb C}(\mathbb C^{s}\oplus (\mathbb {C}(1))^t)$ for some $s,t\ge1$, where $\mathbb C(1)$ is the graded $\mathbb Z_2$-graded $\mathbb C$-module by putting $\mathbb C$ in degree 1. By Theorem \ref{thm-HY}(ii), $C_{A^!}(\theta_f)$ is a strongly $\mathbb Z_2$-graded algebra. Hence $t\neq0$. Therefore $C_{A^!}(\theta_f)_0\cong M_s(\mathbb {C})\times M_t(\mathbb C)$, where $M_s(\mathbb C)$ (resp. $M_t(\mathbb C)$) is the $s\times s$ (resp. $t\times t$) matrix algebra over the field $\mathbb C$. By Theorem \ref{thm-HY}(iii), $\underline{\mcm} A/(f)\cong D^b(C_{A^!}(\theta_f)_0)$ since $C_{A^!}(\theta_f)$ is strongly graded. Since $C_{A^!}(\theta_f)_0$ has two nonisomorphic simple modules, $A/(f)$ has two nonprojective indecomposable graded Cohen-Macaulay modules (up to isomorphisms and degree shifts).

(ii) If $A/(f)$ is a simple graded isolated singularity of 1-type, then $C_{A^!}(\theta_f)$ is a matrix algebra over the $\mathbb Z_2$-graded algebra $\mathbb{CG}$. Then $C_{A^!}(\theta_f)_0$ is a matrix algebra over $\mathbb C$. Therefore $C_{A^!}(\theta_f)_0$ has one nonisomorphic simple module. By Theorem \ref{thm-HY}(iii) again, $A/(f)$ has one nonprojective indecomposable graded Cohen-Macaulay module (up to isomorphisms and degree shifts).
\end{proof}

\begin{theorem}\label{thm-k} Let $A$ and $B$ be noetherian Koszul Artin-Schelter regular algebras, and let $f\in A_2$ and $g\in B_2$ be central regular elements. Suppose that $A\otimes B$ is noetherian.
\begin{itemize}
  \item [(i)] If $B/(g)$ is a simple graded isolated singularity of 0-tpye, then there are equivalences of triangulated categories $$\underline{\mcm}(A\otimes B)/(f+g)\cong D^b(\text{\rm mod}C_{A^!}(\theta_f)_0)\cong\underline{\mcm}A/(f);$$
  \item [(ii)] If $B/(g)$ is a simple graded isolated singularity of 1-type, there is an equivalence of triangulated categories $$\underline{\mcm}(A\otimes B)/(f+g)\cong D^b(\text{\rm mod}C_{A^!}(\theta_f)).$$
\end{itemize}
\end{theorem}
\begin{proof} (i) Since $B/(g)$ is a simple graded isolated singularity of $0$-type, $C_{B^!}(\theta_g)$ is a matrix algebra over the $\mathbb Z_2$-graded algebra $\mathbb C$, and hence is graded Morita equivalent to $\mathbb C$. By Lemma \ref{lem-z2}, $C_{A^!}(\theta_f)\hat\otimes C_{B^!}(\theta_g)$ is graded Morita equivalent to $C_{A^!}(\theta_f)\hat\otimes\mathbb C\cong C_{A^!}(\theta_f)$, and the later one is isomorphic to $C_{A^!}(\theta_f)$. Hence by Theorem \ref{thm-HY}(iii) and Lemma \ref{lem-clifford}, we have equivalences of triangulated categories
\begin{eqnarray*}
  \underline{\mcm}(A\otimes B)/(f+g)&\cong& D^b(\gr_{\mathbb Z_2} C_{(A\otimes B)^!}(\theta_{f+g}))\\
&\cong& D^b\left(\gr_{\mathbb Z_2}\left(C_{A^!}(\theta_f)\hat\otimes C_{B^!}(\theta_g)\right)\right)\\
&\cong& D^b(\gr_{\mathbb Z_2}C_{A^!}(\theta_f))\\
&\cong&\underline{\mcm}A/(f).
\end{eqnarray*}
Since $C_{A^!}(\theta_f)$ is a strongly $\mathbb Z_2$-graded, $D^b(\gr_{\mathbb Z_2}C_{A^!}(\theta_f))\cong D^b(\text{mod}C_{A^!}(\theta_f)_0)$ as triangulated categories. Hence the statement (i) follows.

(ii) As in the proof of (i), we have an equivalence of triangulated categories $$\underline{\mcm}(A\otimes B)/(f+g)\cong D^b\left(\gr_{\mathbb Z_2}\left(C_{A^!}(\theta_f)\hat\otimes C_{B^!}(\theta_g)\right)\right).$$ Since $B/(g)$ is a simple graded isolated singularity of 1-type, $C_{B^!}(\theta_g)$ is a matrix algebra over the $\mathbb Z_2$-graded algebra $\mathbb {CG}$, and hence is graded Morita equivalent to $\mathbb{CG}$. By Lemma \ref{lem-clifford} again, $D^b\left(\gr_{\mathbb Z_2}\left(C_{A^!}(\theta_f)\hat\otimes C_{B^!}(\theta_g)\right)\right)\cong D^b\left(\gr_{\mathbb Z_2}\left(C_{A^!}(\theta_f)\hat\otimes \mathbb {CG}\right)\right)$. By Proposition \ref{prop-str}, $D^b\left(\gr_{\mathbb Z_2}\left(C_{A^!}(\theta_f)\hat\otimes \mathbb {CG}\right)\right)\cong D^b(\text{mod}C_{A^!}(\theta_f))$. Hence the statement (ii) follows.
\end{proof}

\begin{remark}\label{rmk} Theorem \ref{thm-k}(i) may be viewed as a generalization of Kn\"{o}rrer's periodicity theorem (see  \cite[Theorem 3.1]{K}, which has been generalized to noncommutative algebras in \cite{CKMW,HY,MU}). Indeed, let $B=\mathbb{C}[x,y]$ and $g=x^2+y^2$. By Example \ref{ex-s1}, $B/(g)$ is a simple graded isolated singularity of 0-type. Let $A$ be a noetherian Koszul Artin-Schelter regular algebra, and let $f\in A_2$ be a central regular element. The second double branch cover of $A/(f)$ is defined to be the quotient algebra \cite{K,CKMW} $$(A/(f))^{\#\#}=A[x,y]/(f+x^2+y^2).$$ Since $A[x,y]\cong A\otimes B$ is noetherian and $A[x,y]/(f+x^2+y^2)\cong (A\otimes B)/(f+g)$, by Theorem \ref{thm-k}(i), $\underline{\mcm}(A/(f))^{\#\#}\cong\underline{\mcm}A/(f)$.
\end{remark}

\section{Double branch covers of noncommutative conics}
Noncommutative conics in Calabi-Yau quantum planes were recently classified by Hu-Matsuno-Mori in \cite{HMM}. In this section, we will study the double branch covers of the noncommutative conics obtained in \cite[Corollary 3.8]{HMM}.

We say that $A/(f)$ is a \textit{noncommutative conic} (see  \cite[Definition 1.3]{HMM}) if $A$ is a
3-dimensional Calabi-Yau quantum polynomial algebra, and $f\in A_2$ is a central regular element. The double branch cover of $A/(f)$ is defined to be
$$(A/(f))^\#=A[x]/(f+x^2).$$
By Theorem \ref{thm-k}(ii), we also have
$$\underline{\mcm}(A/(f))^\#\cong D^b(\mathrm{mod}C_{A^!}(\theta_f)).$$

If $A/(f)$ is commutative, by \cite[Corollary 3.8(i)]{HMM}, $A/(f)$ is isomorphic to
one of the following algebras:
$$\mathbb{C}[x,y,z]/(x^2),\quad \mathbb{C}[x,y,z]/(x^2+y^2),\quad \mathbb{C}[x,y,z]/(x^2+y^2+z^2),$$
and we have the following table:

\begin{table}[!htbp]
	\centering
	\caption{Commutative case.}
	\begin{tabular}{|c|c|c|}
		\hline
		$A/(f)$&$C_{A^!}(\theta_f)$&$C_{A^!}(\theta_f)_0$\\ \hline
		$\mathbb{C}[x,y,z]/(x^2)$ & $\mathbb{C}_{-1}[x,y,z]/(x^2-1,y^2,z^2)$&$\bigwedge(u,v)$\\ \hline
		$\mathbb{C}[x,y,z]/(x^2+y^2)$&$\mathbb{C}_{-1}[x,y,z]/(x^2-1,y^2-1,z^2)$&$\mathbb{C}_{-1}[u,v]/(u^2-1,v^2)$ \\ \hline
		$\mathbb{C}[x,y,z]/(x^2+y^2+z^2)$&$M_2(\mathbb{CG})$&$M_2(\mathbb{C})$\\
\hline
\end{tabular}
\end{table}

If $A/(f)$ is noncommutative, by \cite[Corollary 3.8(ii)]{HMM}, $A/(f)$ is isomorphic to
$$S^{(\alpha,\beta,\gamma)}/(ax^2+by^2+cz^2)$$
for some $\alpha,\beta,\gamma\in\mathbb{C}$ and $(a,b,c)\in \mathbb{P}^2$, where
$$S^{(\alpha,\beta,\gamma)}=\mathbb{C}\langle x,y,z\rangle/(yz+zy+\alpha x^2, zx+xz+\beta y^2,xy+yx+\gamma z^2).$$
In this case, $C_{A^!}(\theta_f)$ is commutative.

Next, we focus on the noncommutative case. We need some preparations.

Let $G$ be a finite group with identity $e$, and let $E$ be a finited dimensional $G$-graded algebra.
A \textit{$G$-element} $x$ is an invertible homogeneous element of $E$ such that $x^i\mapsto |x|^i$ is an injective group homomorphism from $\langle x\rangle$ to $G$. Here, we use $|x|$ to denote the degree of $x$ and $\langle x\rangle$ is a cyclic group generated by $x$ via multiplication in $E$.

\begin{example}
$\mathbb{C}[x]/(x^n)$ graded by $\mathbb{Z}_n$ has no $\mathbb{Z}_n$-element for positive degrees and $\mathbb{C}G$ graded by $G$ has $G$-elemets for every degree.
\end{example}

\begin{example}
Let $E=M_3(\mathbb{C})$ graded by $\mathbb{Z}_2$, where
$$E_0=\left[\begin{matrix}
\mathbb{C}&\mathbb{C}&0\\
\mathbb{C}&\mathbb{C}&0\\
0&0&\mathbb{C}
\end{matrix}\right],\quad E_1=\left[\begin{matrix}
0&0&\mathbb{C}\\
0&0&\mathbb{C}\\
\mathbb{C}&\mathbb{C}&0
\end{matrix}\right].$$
Then $E$ is strongly graded ($E_1E_1=E_0$), but $E$ has no $\mathbb{Z}_2$-element in degree 1.
\end{example}
\begin{remark}\label{rmk-1}
Let $x\in E$ be a $G$-element. Consider the skew group algebra $E_e\#\langle x\rangle$: as a vector space $E_e\#\langle x\rangle=E_e\otimes \mathbb{C}\langle x\rangle$, and the
 multiplication of $E_e\#\langle x\rangle$ is defined by
$$(a_e\#g)(b_e\#h)=a_e(gb_eg^{-1})\#gh.$$
One sees that $E_e\#\langle x\rangle$ may be regarded as a graded subalgebra of $E$ via $(a_e,x^i)\mapsto a_ex^i$. The
Moreover, if $G \cong\langle x\rangle$, then $E\cong E_e\#\langle x\rangle$.
\end{remark}
\begin{lemma}\label{lem-lift}
Assume that $I$ is a nilpotent homogeneous ideal of $E$, then any $G$-element $\overline{x}$ of $E/I$ can be lifted into $E$.
\end{lemma}

\begin{proof}
Assume $x\in E$ is a homogemeous preimage of $\overline{x}$. Since $\overline{x}^n=\overline{1}$ for some $n\ge0$, there is an element $r\in I$ such that $x^n=1+r$. Note that $r$ is a nilpotent element of degree $e$, so $1+r$ is invertible.

More precisely, since $|x^n|=e$, $r=x^n-1\in I_e$ and commutative with $x$. If $I^2=0$, then $x':=x(1-r/n)$ is a lifting of $\overline{x}$. In fact, $x'$ is homogeneous,
$$x'^n=x^n(1-r/n)^n=(1+r)(1-r)=1$$
and $x-x'=(x/n)r\in I$.

 For general case, the result follows from the induction on the exact sequence
$$0\rightarrow I^{2^{k-1}}/I^{2^k}\rightarrow E/I^{2^k}\rightarrow E/I^{2^{k-1}}\rightarrow0.$$
\end{proof}

\begin{proposition}\label{prop-copy}
Let $E$ be a finite dimensional commutative algebra graded by $\mathbb{Z}_2$, if $A$ is strongly graded, then $A\cong A_0\times A_0$ as ungraded algebras.
\end{proposition}

\begin{proof}
Since $E$ is finite dimensional, the graded radical $J^g(E)$ is nilpotent and $E/J^g(E)$ is graded semisimple. The condition that $E$ is strongly graded implies that $E/J^g(E)$ is also strongly graded, hence $E/J^g(E)$ is a product of some copies of $\mathbb{CG}$ by Lemma \ref{lem-z2}(ii) and has a $\mathbb{Z}_2$-element in degree 1. By Lemma \ref{lem-lift}, this $\mathbb{Z}_2$-element can be lifted into $E$, denoted by $g$. Now, $E_0\#\langle g\rangle$ is a subalgebra of $E$, and as ungraded algebras
$$E_0\#\langle g\rangle\cong E_0[x]/(x^2-1)\cong E_0\times E_0.$$
For surjective, note that if $a_1\in E_1$, then $a_1=(a_1g)g\in E_0\#\langle g\rangle$.
\end{proof}

\begin{corollary}\label{cor}
Let $A/(f)$ be a noncommutative conic. If $A/(f)$ is noncommutative, then
\begin{align*}
\underline{\mcm}(A/(f))^{\#}\cong \underline{\mcm}(A/(f))\times \underline{\mcm}(A/(f)).
\end{align*}
\end{corollary}

\begin{example}
Let $A=\mathbb{C}\langle x,y,z\rangle/(yz+zy+x^2,zx+xz+y^2,xy+yx)$ and $f=3x^2+3y^2+4z^2$. Then
$$A^!=\mathbb{C}[x,y,z]/(yz-x^2,zx-y^2,z^2),$$
$$C_{A^!}(\theta_{f})=\mathbb{C}[x,y,z]/(yz-x^2+3,zx-y^2+3,z^2-4).$$
Since $(yz-x^2+3)-(zx-y^2+3)=0$, we have $(x-y)(x+y+z)=0$. One can compute that $\mathrm{Spec}(C_{A^!}(\theta_{f}))$ has $4$ points and $C_{A^!}(\theta_{f})$ is a product of $4$ commutative local rings.

Let $R=\mathbb{C}[z]/(z^2-4)$, then
$$C_{A^!}(\theta_{f})\cong R[x,y]/(yz-x^2+3,zx-y^2+3)\cong A'\times A''$$
where $A'=\mathbb{C}[x,y]/(2y-x^2+3,2x-y^2+3)$ and $A''=\mathbb{C}[x,y]/(-2y-x^2+3,-2x-y^2+3)$.

In $A'$, we have
$$y^4=(2x+3)^2=4x^2+12x+9=4(2y+3)+6(y^2-3)+9=6y^2+8y+3.$$
Hence, $y^4-6y^2-8y-3=(y+1)^3(y-3)=0$, which implies $A'\cong \mathbb{C}[u]/(u^3)\times \mathbb{C}$. A similar computation shows that $A'\cong A''$. Therefore, $C_{A^!}(\theta_{f})_0$ is also isomorphic to $\mathbb{C}[u]/(u^3)\times \mathbb{C}$ by Propsition \ref{prop-copy}.

By Remark \ref{rmk-2}, $\mathrm{rank}f$ must be 1. In fact,
$f=3x^2+3y^2+4z^2=(-x-y+2z)^2.$
\end{example}

\vspace{5mm}

\subsection*{Acknowledgments}
J.-W. He was supported by NSFC (No. 11971141). Y. Ye was supported by NSFC (No. 11971449).

\vspace{5mm}


\end{document}